\documentclass[reqno]{amsart}
\usepackage{amssymb,amsmath,hyperref}
\usepackage{amsrefs, float}
\usepackage[foot]{amsaddr}
\usepackage{bbold,stackrel, multicol}
 
\newcommand{\Z}{{\textsf{\textup{Z}}}}

\newtheorem{thm}{Theorem}

\newtheorem{cor}[thm]{Corollary}
\newtheorem{defi}[thm]{Definition}

\newtheorem{rem}[thm]{Remark}
\newtheorem{nota}[thm]{Notation}

\newtheorem{princ}[thm]{Principle}

\newtheorem{ack}[thm]{Acknowledgement}

\newtheorem*{tempo*}{Template}

\newtheorem{remark}[thm]{Remark}

\newcommand\be{\begin{equation}}
\newcommand\ee{\end{equation}} 

\usepackage{amsmath,amsfonts} 

\usepackage[applemac]{inputenc}

\def\bdefi{\begin{defi}\rm}
\def\edefi{\end{defi}}
\def\bnota{\begin{nota}\rm}
\def\enota{\end{nota}}

\def\FIVE{\Pi_{1}^{1}\text{-\textup{\textsf{CA}}}_{0}}

\def\SIXko{\Pi_{k+1}^{1}\text{-\textsf{\textup{CA}}}_{0}}
\def\SIXK{\Pi_{k}^{1}\text{-\textsf{\textup{CA}}}_{0}^{\omega}}

\def\ATR{\textup{\textsf{ATR}}}

\def\ZF{\textup{\textsf{ZF}}}

\def\L{\textsf{\textup{L}}}

\def\RCA{\textup{\textsf{RCA}}}
\def\({\textup{(}}
\def\){\textup{)}}

\def\RCAo{\textup{\textsf{RCA}}_{0}^{\omega}}
\def\ACAo{\textup{\textsf{ACA}}_{0}^{\omega}}

\def\WKL{\textup{\textsf{WKL}}}

\def\N{{\mathbb  N}}
\def\Q{{\mathbb  Q}}
\def\R{{\mathbb  R}}

\def\SS{\textup{\textsf{S}}}

\def\di{\rightarrow}

\def\asa{\leftrightarrow}

\def\ACA{\textup{\textsf{ACA}}}

\def\QFAC{\textup{\textsf{QF-AC}}}

\def\ran{\textup{\textsf{ran}}}

\def\SIND{\Sigma\textup{\textsf{-IND}}}

\def\cocode{\textup{\textsf{cocode}}}

\def\NIN{\textup{\textsf{NIN}}}

\def\open{\textup{\textsf{open}}}
\def\enum{\textup{\textsf{enum}}}

\def\CBT{\textup{\textsf{CBT}}}

\def\BOOT{\textup{\textsf{BOOT}}}

\def\RANGE{\textup{\textsf{RANGE}}}
\def\range{\textup{\textsf{range}}}

\def\PST{\textup{\textsf{PST}}}
\def\eps{\varepsilon}

\def\CC{\textup{\textsf{CC}}}

\def\INT{\textup{\textsf{int}}}

\usepackage{mathrsfs}  

\usepackage{graphicx}
\usepackage{tikz}
\usetikzlibrary{matrix, shapes.misc}
\usepackage{comment,tikz-cd}

\numberwithin{equation}{section}
\numberwithin{thm}{section}

\begin{document}
\title{Coding is non-robust}
\author{Sam Sanders}
\address{Department of Philosophy II, RUB Bochum, Germany}
\email{sasander@me.com}


\begin{abstract}
For various reasons, higher-order objects are often studied in mathematical logic via second-order `codes' or `representations'.  
An important example hailing from real analysis and topology is provided by open sets, which are represented by unions $\cup_{n\in \N}I_{n}$ where each $I_{n}$ is a basic open interval.  
Now, Montalb\'an has recently highlighted the importance of robustness of logical systems in Reverse Mathematics (abbreviated RM; see \cite{simpson2,montahue}).  It is then a natural RM-question whether basic properties of open sets are robust under slight modifications of the coding of open sets.  Here, we study the representation where $I_{n}$ as above is \emph{either} an open interval \emph{or} the union of two such intervals \emph{but} we cannot decide which one.  
Under this slight variation of the usual coding, basic properties of open and closed sets readily imply the relatively strong system $\ATR_{0}$ from RM.  
Moreover, we obtain equivalences for the former properties and the \emph{enumeration principle}.  The latter states that countable sets can be enumerated and boasts many equivalences from Fourier analysis.  
Along the way, we investigate the RM-properties of the closure, interior, and boundary of sets of reals, a study interesting in its own right.  
\end{abstract}
\setcounter{page}{0}
\tableofcontents
\thispagestyle{empty}
\newpage
\maketitle              

\section{Introduction}
\subsection{Aim and motivation}\label{intro}
In a nutshell, we introduce a slight variation of the usual representation of open sets in mathematical logic and show that basic properties of such open sets readily imply the relatively strong system $\ATR_{0}$.  To this end, we investigate the logical properties of the topological closure, interior, and boundary, a study which turns out to be of independent interest. 

\smallskip

In more detail, the notion of open (and closed) set is uniquely central to mathematics, especially topology and analysis.
Historically, this concept dates back to Baire's doctoral thesis (\cite{beren2}) while Dedekind already studied such notions twenty years earlier; the associated paper \cite{didicol} was only published much later (\cite{moorethanudeserve}).
Now, for various reasons, the study of open sets in mathematical logic usually takes place indirectly, namely via \emph{representations} or \emph{codes} for open sets, as discussed in detail in Remark \ref{XzX}.
In this light, it is a natural question whether this representation is \emph{robust} or whether slight variations of this representation yield different results.  The importance of robustness in Reverse Mathematics (abbreviated RM, see \cite{simpson2, stillebron, damurm, samBOOK} for an introduction) has been qualified as follows.  
 \begin{quote}
[\dots] one distinction that I think is worth making is the one between robust systems and non-robust systems. 
A system is \emph{robust} if it is equivalent to small perturbations of itself. (\cite{montahue}*{p.\ 432})
\end{quote}
In this paper, we study the robustness of the well-known representation of open sets of reals as countable unions of open intervals.  
We introduce the $D$-representation (Definition \ref{charkar}) where an open set is given as a union $\cup_{n\in \N}J_{n}$ where each $J_{n}$ is \emph{either} an open interval \emph{or} the union of two open intervals \emph{but} we cannot tell which one always.  
We show that basic properties of $D$-open sets readily imply $\ATR_{0}$.  

\smallskip

To this end, we obtain equivalences involving properties of $D$-open sets and the \emph{enumeration principle} from \cite{dagsamXI, samBIG} in Kohlenbach's higher-order RM, which is sketched in Section~\ref{ohoam}.  In turn, this study requires RM-results concerning the topological boundary, interior, and closure of sets, interesting in their own right.  A brief overview of our results is as follows. 
\begin{itemize} 
\item We study the RM-properties of the existence of the topological boundary, closure, and interior of sets of reals as follows in Section \ref{setbou}. 
\begin{itemize}
\item For closed and open sets, the boundary, closure, and interior exist as sets assuming arithmetical comprehension (Theorem \ref{talkingaboud}).  
\item For $\bf F_{\sigma}$-sets, the existence of the boundary is not provable in rather strong systems (Theorem \ref{justnin}) and implies $\ATR_{0}$ (Theorem \ref{moare}).  
\item The existence of the interior for general sets of reals is equivalent to Feferman's projection principle (Theorem \ref{tenar}), one of the strongest axioms studied in higher-order RM. 
\end{itemize}
\item Basic properties of $D$-open sets are equivalent to the \emph{enumeration principle} (Theorem \ref{trefor}), which states that countable sets can be enumerated and implies $\ATR_{0}$ (\cite{dagsamX}).  
Examples include the following theorems. 
\begin{itemize}
\item Urysohn's lemma for disjoint $D$-closed sets of reals. 
\item The supremum and maximum principle for continuous (and other classes of) functions on $D$-closed sets in $[0,1]$. 
\item A $D$-open set of reals has a second-order code.  
\item The perfect set theorem for the $D$-representation: a non-enumerable $D$-closed set contains a $D$-closed and perfect set.  
\end{itemize}
\item We discuss related results and the associated foundational implications.
\begin{itemize}
\item Replacing `enumerable' by `injection to $\N$' in the perfect set theorem and Cantor-Bendixson theorems results in principles that do not imply $\ATR_{0}$. 
This follows from our answer to a question of Mummert posed during the 2022 RM-meeting in Paris (Section \ref{mummershines}). 
\item We draw a parallel between the recent history of analysis and the coding of open sets in second-order RM (Section \ref{fasp1}).  
\item We discuss the non-robustness of representing open sets in the grand scheme of things (Section \ref{fasp2}).  
\end{itemize}
\end{itemize}
Finally, open sets are often studied indirectly via `codes' or `representations' in mathematical logic. 
The following historical remark provides an overview of some such approaches to open sets.
\begin{rem}[Open sets and representations]\label{XzX}\rm
A set is \emph{open} if it contains a neighbourhood around each of its points, and an open set can be represented as a \emph{countable} union of such neighbourhoods in separable spaces, a result going back to Cantor.    
The latter and similar representations are used in various `computational' approaches to mathematics, as follows. 

\smallskip

For instance, the neighbourhood around a point of an open set is often assumed to be given together with this point (see e.g.\ \cite{bish1}*{p.\ 69}).  This is captured by the R2-representation studied in \cite{dagsamVII}*{\S7} and below.
Alternatively, open sets are simply represented as countable unions -called `codes' in \cite{simpson2}*{II.5.6}, `names' in \cite{wierook}*{\S1.3.4}, `witnesses' in \cite{weverketoch}, and `presentations' in \cite{littlefef}- of basic open neighbourhoods.
A related notion is \emph{locatedness} which means the (continuous) distance function $d(a, A):=\inf_{b\in A}d(a, b)$ exists for the set $A\subset \R$ (see \cite{bish1}*{p.\ 82}, \cite{withgusto}, or \cite{twiertrots}*{p.\ 258}), and numerous sufficient conditions are known (\cite{withgusto}).    

\smallskip

The representation of open sets as countable unions has the advantage that one can (effectively) switch between codes for open sets and codes for continuous functions (see e.g.\ \cite{simpson2}*{II.7.1}).  
As shown in \cite{dagsamVII, kohlenbach4, dagsamXIV}, the base theory plus $\WKL_{0}$ establishes that a continuous function has a code on Cantor space or the unit interval.  
By contrast, the base theory plus full second-order arithmetic cannot prove that an open set has a code (\cite{dagsamVII}). 
\end{rem}

\subsection{On higher-order arithmetic}\label{ohoam}
We assume basic familiarity with Kohlenbach's higher-order RM, the base theory $\RCAo$ in particular. 
The original text is \cite{kohlenbach2} while more recent introductions are in \cites{dagsamXIV, sammetric}. 
A monograph on this topic is forthcoming (\cite{samBOOK}) while a brief sketch is provided next.

\smallskip

Zeroth of all, higher-order RM is officially formulated using types rather than sets.  
Hence, both `$n\in \N$' and `$n^{0}$' are used for `$n$ is a natural number'.  
Similarly, elements of Baire space are identified as `$f^{1}$' or `$f\in \N^{\N}$', while mappings from Baire space to Baire space are denoted $\Phi^{1\di 1}$ or $\Phi:\N^{\N}\di \N^{\N}$.  
We will not really use objects of higher rank in this paper and the use of types is kept to a minimum.  We do introduce (standard) sequence notation in the below (Notation \ref{taxio}).

\smallskip

First of all, real numbers and real equality `$=_{\R}$' are defined in $\RCAo$ in the same way as in second-order RM, i.e.\ as fast-converging Cauchy sequences.  
A function from reals to reals, denoted $F:\R\di \R$, is then given by $\Phi:\N^{\N}\di \N^{\N}$ such that 
\[
(\forall x, y\in \R)(x=_{\R}y\di \Phi(x)=_{\R} \Phi(y)).
\]
Secondly, we consider the following essential axiom where the functional $E$ is also called \emph{Kleene's quantifier $\exists^{2}$} and is discontinuous on Baire space:
\be\tag{$\exists^{2}$}
(\exists E:\N^{\N}\di \{0,1\})(\forall f\in\N^{\N})( E(f)=0 \asa (\exists n\in \N)(f(n)=0)).
\ee
We write $\ACAo\equiv \RCAo+(\exists^{2})$ and observe that the latter proves the same second-order sentences as $\ACA_{0}$ (see \cite{hunterphd}).
We shall mostly work in $\ACAo$, which is convenient as set of reals can then given via characteristic functions.  

\smallskip

The following definition of sets aims at maximal compatibility with the second-order framework following \cite{simpson2}*{II.7.1}.  
\bdefi[Sets]\label{char}~
\begin{enumerate}
\renewcommand{\theenumi}{\alph{enumi}}
\item A set of reals $A$ and its complement $A^{c}$ are given by a representation function $F_{A}:\R\di \R^{+}\cup\{0\}$; we write `$x\in A$' for $ F_{A}(x)>_{\R}0$ and `$x \in A^{c}$' if $F_{A}(x)=_{\R}0$.\label{tinkzzz}
\item We write `$A\subseteq B$' if we have $(\forall x\in \R)(x\in A\di x\in B)$.  
\item A set $O\subseteq \R$ is \emph{open} in case $x\in O$ implies that there is $k\in \N$ such that $B(x, \frac{1}{2^{k}})\subseteq O$.\label{qzopen}  
\item The complement $O^{c}$ of an open set $O\subset \R$ is called closed.  
\item A set $O\subseteq \R$ is \emph{RM-open} if there are sequences $(a_{n})_{n\in \N}, (b_{n})_{n\in \N}$ of reals such that $x\in O$ if and only if $x\in \cup_{n\in \N}(a_{n}, b_{n})$ for all $x\in \R$ (\cite{simpson2}*{II.5.6}).\label{daz}
\item The complement $O^{c}$ of an RM-open set $O\subset \R$ is called RM-closed.  
\item A set $A\subset \R$ has \emph{measure zero} if for any $\eps>0$ there is a sequence of open intervals $(I_{n})_{n\in \N}$ such that $\cup_{n\in \N}I_{n}$ covers $A$ and $\eps>\sum_{n=0}^{\infty}|I_{n}|$. 
\item A set $A\subset \R$ is \emph{enumerable} if there is a sequence $(x_{n})_{n\in \N}$ that includes all elements of $A$. 
\item A set $A\subset \R$ is \emph{countable} if there is $Y:\R\di \N$ satisfying
\[
(\forall x, y\in A)(Y(x)=_{\N}Y(y)\di x=_{\R}y),
\]
where $Y$ is called `injective on $A$'.  
\end{enumerate}
\edefi
With this convention in place, we can adopt the usual definitions of nowhere dense set, meagre set, $\bf{F}_{\sigma}$-set, the Baire property, et cetera from the literature. 
We do feel the need to recall some standard sequence notation from type theory.
\begin{nota}\label{taxio}\rm
Finite sequences of naturals are denoted $\sigma\in \N^{<\N}$ or $\sigma^{0^{*}}$ with the empty sequence being `$\langle \rangle$'.  
We do not always distinguish between a natural number $n$ and the associated one-element sequence $\langle n\rangle$. 
For $\sigma^{0^{*}}$ and $\tau^{0^{*}}$, the `concatenation of $\sigma$ followed by $\tau$' is denoted $\sigma*\tau$. 
For $\sigma^{0^{*}}=\langle n_{0}, \dots n_{k}\rangle$, the `length $|\sigma|$' is $k+1$ while $\overline{\sigma}m$ for $m<|\sigma|$ is $\sigma$ `cut off' after the first $m+1$ elements.  
We assume a standard pairing function is given, which codes finite sequences as natural numbers.  
We use similar notations for infinite sequences of naturals with their obvious meaning. 
\end{nota}
Thirdly, following Notation \ref{taxio}, consider the following axiom where the functional $\SS^{2}$ is often called \emph{the Suslin functional} (\cite{kohlenbach2, avi2, yamayamaharehare}):
\be\tag{$\SS^{2}$}
(\exists\SS:\N^{\N}\di \{0,1\})(\forall f \in \N^{\N})\big[  (\exists g \in \N^{\N})(\forall n \in \N)(f(\overline{g}n)=0)\asa \SS(f)=0  \big].
\ee
By definition, the Suslin functional $\SS^{2}$ can decide whether a $\Sigma_{1}^{1}$-formula in normal form, i.e.\ as in the left-hand side of $(\SS^{2})$, is true or false.   
The system $\FIVE^{\omega}\equiv \RCAo+(\SS^{2})$ proves the same second-order sentences as $\FIVE$ via a straightforward variation of the conservation results from \cite{hunterphd}.  

\smallskip

We also consider the functional $\SS_{k}^{2}$ which decides the truth or falsity of $\Sigma_{k}^{1}$-formulas in normal form; we define 
the system $\SIXK$ as $\RCAo+(\SS_{k}^{2})$, where  $(\SS_{k}^{2})$ expresses that $\SS_{k}^{2}$ exists.  
We define $\Z_{2}^{\omega}$ as $\cup_{k}\SIXK$ as one possible higher-order version of $\Z_{2}$.
The functionals $\nu_{n}$ from \cite{boekskeopendoen}*{p.\ 129} are essentially $\SS_{n}^{2}$ strengthened to return a witness (if existent) to the $\Sigma_{n}^{1}$-formula at hand.  
The operator $\nu_{n}$ is Hilbert-Bernays' $\nu$ from \cite{hillebilly2}*{p.\ 479} restricted to $\Sigma_{n}^{1}$-formulas. 

\smallskip

\noindent
Fourth, we introduce Kleene's quantifier $\exists^{3}$ as follows:
\be\tag{$\exists^{3}$}
(\exists E: (\N^{\N}\di \N)\di \N)(\forall Y:\N^{\N}\di \N)\big[  (\exists f \in \N^{\N})(Y(f)=0)\asa E(Y)=0  \big].
\ee
Both $\Z_{2}^{\Omega}\equiv \RCAo+(\exists^{3})$ and $\Z_{2}^\omega\equiv \cup_{k}\SIXK$ are conservative over $\Z_{2}$ as shown in \cite{hunterphd}.
The functional from $(\exists^{3})$ is also called `Kleene's quantifier $\exists^{3}$', and we use the same convention for other functionals.  Hilbert-Bernays' operator $\nu$ from \cite{hillebilly2}*{p.\ 479} is essentially Kleene's $\exists^{3}$, modulo a non-trivial fragment of the Axiom of Choice.    

\smallskip

Fifth, like its second-order counterpart, the development of higher-order RM sometimes takes place over an \emph{extension} of the base theory.  
The following fragment of countable choice, not provable\footnote{One readily proves that $\QFAC^{0,1}$ is equivalent to $\CC(\R)$, i.e.\ countable choice for countable unions of sets of reals, over $\ZF$ (\cite{heerlijkheid}).} in $\ZF$, plays an important role.
\begin{princ}[$\QFAC^{0,1}$]
Let $\varphi$ be quantifier-free with $(\forall n\in \N)(\exists f\in \N^{\N})\varphi(f, n)$, then there exists a sequence $  (f_{n})_{n\in \N}$ in $\N^{\N}$ with $(\forall n\in \N)\varphi(f_{n}, n)$.
\end{princ}
The local equivalence between sequential and `epsilon-delta' continuity cannot be proved in $\ZF$ (\cite{heerlijkheid}), but can be established in $\RCAo+\QFAC^{0,1}$, as first shown by Kohlenbach in \cite{kohlenbach2}.  
In this light, it should not be a surprise that $\RCAo+\QFAC^{0,1}$ often occurs as a base theory.  Similarly, extra induction axioms are sometimes used in second-order RM (\cite{neeman}) and we will often use the following fragment of induction, where formulas of the same form as $\varphi(n)$ are called\footnote{The usual formula hierarchy based on the classes $\Sigma_{k}^{i}$ for $i=0, 1$ only allows for first- and second-order parameters.  
As a generalisation of $\Sigma_{1}^{1}$ to higher-order parameters, we let a \emph{$\Sigma$-formula} be any formula of the form $(\exists f\in 2^{\N})(Y(f, n)=0)$ for third-order $Y$. } a \emph{$\Sigma$-formula}; the negation of the latter is called a \emph{$\Pi$-formula}. 
\begin{princ}[$\SIND$]  
The induction axiom for formulas of the form $\varphi(n)\equiv (\exists f\in \N^{\N})(Y(f, n)=0)$ for any $Y^{2}$.   
\end{princ}
Finally, it is an empirical observation that, on one hand, many third-order theorems about (possibly) discontinuous functions are equivalent to the Big Five of RM (\cite{dagsamXIV}), 
while on the other hand many third-order theorems are provable in $\Z_{2}^{\Omega}$ but not in $\Z_{2}^{\omega}+\QFAC^{0,1}$.  
The weakest `natural' theorem not provable in $\Z_{2}^{\omega}+\QFAC^{0,1}$ is well-known as the \emph{uncountability of the reals}, formulated as follows:
\begin{center} 
$\NIN_{[0,1]}$: there is no injection from $[0,1]$ to $\N$.  
\end{center}
A long list of theorems that imply $\NIN_{[0,1]}$ can be found in \cites{dagsamX, samBOOK, samBIG} while a stronger negative result is proved in \cite{dagsamXVII}. 

\section{Setting boundaries}\label{setbou}
\subsection{Introduction}
We study the RM-properties of the existence of the boundary, closure, and interior of sets of reals. 
In particular, we show that these constructs exist for open and closed sets using only arithmetical comprehension 
while the generalisation to $\bf F_{\sigma}$-sets already implies $\ATR_{0}$.  The general case yields Feferman's projection principle and Kleene's quantifier $(\exists^{3})$, which are the strongest axioms studied in higher-order RM. 
We shall discuss foundational implications in Section~\ref{fasp1} but do feel the need to mention the following quote \emph{hic et nunc}.  
\begin{quote}
Most mathematics naturally lies within the realm of complete separable metric
spaces and continuous functions between them defined on open, closed, compact or
$G_{\delta}$ subsets. (\cite{fried5})
\end{quote}
First of all, concepts like boundary are well-known from topology and defined as follows based on Definition \ref{char}.
\begin{defi} Let $E\subset \R$ be an arbitrary set.  
\begin{itemize} 
\item A real $x\in \R$ is a \emph{limit point} of $E$ if $(\forall k\in \N)(\exists y\in B(x, \frac{1}{2^{k}}))(y\ne x\wedge y\in E)$. 
\item The \emph{closure} $\overline{E}$ is the set $E$ together will all its limit points.
\item The \emph{interior} $\INT(E)$ is $\{x\in E: (\exists N\in \N)( B(x, \frac{1}{2^{N}})\subset E)\} $.  
\item The \emph{boundary} $\partial E$ is the difference $\overline{E}\setminus \INT(E)$.
\end{itemize}
\end{defi}
As the associated definitions quantify over the underlying space, these constructs do not exist in general as sets, even in fairly strong logical systems.   
In particular, the general existence of the interior is equivalent to \emph{Feferman's projection principle} by Theorem \ref{tenar}.   
Regarding the latter, the late Sol Feferman was a leading Stanford logician with a life-long interest in foundational matters (\cite{fefermanlight, fefermanmain}), esp.\ so-called predicativist mathematics following Russell and Weyl (\cite{weyldas}). 
As part of this foundational research, Feferman introduced the `projection principle' \textsf{Proj}$_{1}$ in \cite{littlefef}, a third-order version of $(\exists^{3})$.
We study the principle $\BOOT$, which is Feferman's principle \textsf{Proj}$_{1}$ in $\L_{\omega}$.
Kohlenbach's principle $\Pi_{1}^{1,b}$-$\textsf{CA}_{0}$ from \cite{kohlenbach4}*{\S5} is readily seen to be equivalent to $\BOOT$ over $\ACAo$. 
\begin{princ}[$\BOOT$] For $Y^{2}$, there is $X\subset \N$ such that 
\be\label{EZ}
(\forall n\in \N)\big[n\in X\asa (\exists f\in \N^{\N})(Y(f, n)=0)  \big].
\ee
\end{princ}
The name of this principle derives from the verb `to bootstrap'.  Indeed, $\BOOT$ is third-order and weaker than $(\exists^{3})$, but we still have that $\SIXK+\BOOT$ proves $\SIXko$, which is readily proved.

\subsection{Some equivalences}\label{pluffy}
We show that for closed and open sets, the boundary, closure, and interior exist (as sets) in $\ACAo$ (Theorem \ref{talkingaboud}).  
By contrast, for $\bf F_{\sigma}$-sets, the existence of the boundary is not provable in $\Z_{2}^{\omega}+\QFAC^{0,1}$ (Theorem \ref{justnin}), and even implies $\ATR_{0}$ assuming $\ACAo$ (Theorem \ref{moare}).  
The general existence of the interior is equivalent to Feferman's projection principle $\BOOT$ (Theorem \ref{tenar}).  
The existence of the \emph{interior operator}, mapping a set to its interior, is equivalent to Kleene's quantifier $(\exists^{3})$ over $\ACAo$ (Corollary \ref{fefequo}). 

\smallskip

First of all, to our own surprise, the boundary, closure, and interior of open and closed sets are readily defined, despite the definition involving a quantifier over $\R$. 
\begin{thm}[$\RCAo$]\label{talkingaboud}
The following are equivalent to $(\exists^{2})$. 
\begin{enumerate}
\renewcommand{\theenumi}{\alph{enumi}}
\item For any closed $C\subset \R$, the boundary $\partial{C}$ exists.\label{plifn1}
\item For any open $O\subset \R$, the boundary $\partial{O}$ exists. \label{plifn4}
\item There exists $\mathfrak{B}:(\R\di \R)\di \R$ such that for any closed $C\subset \R$, $\mathfrak{B}(C)$ equals the boundary $\partial C$. \label{plifn5}
\end{enumerate}
The following are provable from $(\exists^{2})$. 
\begin{enumerate}
\renewcommand{\theenumi}{\alph{enumi}}
\setcounter{enumi}{3}
\item For any open $O\subset \R$, the closure $\overline{O}$ exists and is RM-closed.\label{plifn2}
\item For any closed $C\subset \R$, the interior $\textsf{\textup{int}}(C)$ exists and is RM-open. \label{plifn3}
\item For any open $O\subset \R$, the distance function $d(x, O)=\inf_{y\in O}|x-y|$ exists.\label{treffoil}
\end{enumerate}
\end{thm}
\begin{proof}
To obtain $(\exists^{2})$ from item \eqref{plifn1}-\eqref{plifn5}, observe that the boundary of $[0,1]$ and $(0,1)$ is $\{0,1\}$; these intervals (trivially) have continuous representation functions (\cite{simpson2}*{II.7.1}), but $\{0,1\}$ clearly (only) has a discontinuous representation function.  
The existence of a discontinuous function implies $(\exists^{2})$ by \cite{kohlenbach2}*{Prop.\ 3.14}.  

\smallskip

To prove items \eqref{plifn1}-\eqref{plifn3} from $(\exists^{2})$, let $C\subset [0,1]$ be closed and observe that 
for all $a, b\in \R$, we have 
\[
(a, b)\subset C \asa \big[ [(a, b)\cap \Q]\subset C  \big], 
\]
as the complement of $C$ is open.  In this light, `$x\in \textsf{int}(C)$' is equivalent to  
\be\label{zefrepo}\textstyle
x\in C\wedge (\exists N\in \N)(\forall q\in B(x, \frac{1}{2^{N}})\cap \Q)(q\in C)  ),
\ee
which is arithmetical and hence decidable using $(\exists^{2})$.  Thus, $\textsf{int}(C)$ exists as a set and the boundary $\partial C$ is simply the set $C\setminus \textsf{int}(C)$.  
An RM-code of $\textsf{int}(C)$ is $\cup_{r\in \textsf{int}(C)\cap \Q}B(r, R(r))$ where $R(x)$ is the least $N$ as in the second conjunct of \eqref{zefrepo}.  
A similar argument works for open sets and their closures.  To obtain item~\eqref{treffoil}, observe that $d(x, O)= \inf_{y\in O\cap \Q} |x-y|$ for open $O\subset \R$, where $(\exists^{2})$ readily provides the latter infimum. 
\end{proof}
Secondly, going slightly beyond open and closed sets, the boundary is already non-trivial as $\Z_{2}^{\omega}+\QFAC^{0,1}$ cannot prove $\NIN_{[0,1]}$. 
\begin{thm}[$\ACAo+\QFAC^{0,1}$]\label{justnin}
The principle $\NIN_{[0,1]}$ follows from the following closure properties.
\begin{enumerate}
\renewcommand{\theenumi}{\alph{enumi}}
\item For any set $X\subset \R$ that is $\bf F_{\sigma}$, the closure $\overline{X}$ exists.\label{salmu}
\item For any set $X\subset \R$ that is $\bf G_{\delta}$, the interior $\textsf{\textup{int}}{(X)}$ exists.
\end{enumerate}
\end{thm}
\begin{proof}
For the first item, let $Y:[0,1]\di \N$ be an injection.  
Now define the set $B\subset \R$ as follows: $x\in B$ if and only if
\be\label{examp}
 (\exists n\in \N, q\in \Q)\big(n+1<x-q<n+2 \wedge  Y(x-n-1-q)=n\big).
\ee
Let $(q_{n})_{n\in \N}$ be an enumeration of $\Q\cap [0,1]$ and observe that $B=\cup_{n,m\in \N}C_{n, m}$ with 
\[
C_{n, m}:=
 \{x \in \R^{+} : n+1<x-q_{m}<n+2\wedge  Y(x-n-1-q_{m})=n \}   
\]
being a closed set.  Indeed, since $Y:[0,1]\di \N$ is injective, $C_{n,m}$ is at most a singleton.   Now consider the following readily-proved equivalence:
\be\label{temupemu}
(\exists x\in [0,1])(Y(x)=n)\asa (\exists q\in (n, n+1)\cap \Q)( q\in \overline{B} ). 
\ee
Since the right-hand side of \eqref{temupemu} is arithmetical, we can define $\ran(Y)$, i.e.\ the range of $Y$ on the unit interval.  Now apply $\QFAC^{0,1}$ to 
\[
(\forall n\in \N)(\exists x\in [0,1])[n\in \ran(Y)\di Y(x)=n  ]
\]
to obtain an enumeration of the unit interval, contradicting \cite{simpson2}*{II.4.9}, which implies that the unit interval cannot be enumerated.  The second item follows in the same way.
\end{proof}
\noindent
For the following theorem, note that $\ACAo$ is conservative over $\ACA_{0}$ (see \cites{hunterphd, samBOOK}). 
\begin{thm}\label{moare}
The system $\ACAo$ plus item \eqref{salmu} from Theorem \ref{justnin} proves $\ATR_{0}$ and is a conservative extension of the latter. 
\end{thm}
\begin{proof}
For the first part, let $\varphi$ be arithmetical and such that $(\forall n\in \N)(\exists \textup{ at most one } f \in 2^{\N})\varphi(n, f)$. 
Since $(\exists^{2})$ is given, we can decide if $(\exists \sigma \in 2^{<\N})\varphi(n, \sigma*11\dots)$ or $(\exists \sigma \in 2^{<\N})\varphi(n, \sigma*00\dots)$; we may thus assume this case does not occur.  
Now define the set $D\subset \R$ as follows: $x\in D$ if and only if
\[
 (\exists n\in \N, q\in \Q)\big(n+1<x-q<n+2 \wedge  \varphi(n, \eta(x-n-1-q))\big),
\]
where $\eta:[0,1]\di 2^{\N}$ outputs the binary representation of the input, with a tail of zeros if applicable.  
Note that $\eta$ is readily defined via the usual interval-halving technique using $(\exists^{2})$. 
Let $(q_{n})_{n\in \N}$ be an enumeration of $\Q\cap [0,1]$ and observe that $D=\cup_{n,m\in \N}D_{n, m}$ with 
\[
D_{n, m}:=
 \{x \in \R^{+} : n+1<x-q_{m}<n+2\wedge  \varphi(n, \eta(x-n-1-q_{m})) \}   
\]
being a closed set as $D_{n,m}$ is at most a singleton.   Now consider the following readily-proved equivalence:
\be\label{temupemu2}
(\exists f\in 2^{\N})\varphi(n, f)\asa (\exists q\in (n, n+1)\cap \Q)( q\in \overline{D} ). 
\ee
Since the right-hand of \eqref{temupemu2} is arithmetical, there is $X\subset \N$ with $n\in X\asa (\exists f\in 2^{\N})\varphi(n, f)$.  The associated `at most one' comprehension principle implies $\ATR_{0}$ by \cite{simpson2}*{V.5.2}, and we are done.

\smallskip

For the second part, it suffices to establish item \eqref{salmu} in $\ACAo+\open$ as the latter is conservative over $\ATR_{0}$ (see \cite{hunterphd, pastebee}).  
A detailed proof of this claim is found in \cite{samBOOK}*{I.2.4}. 
Now let $A\subset \R$ be $\bf F_{\sigma}$, i.e.\ $A=\cup_{n\in \N}F_{n}$ where each $F_{n}$ is closed.  
Using $\open$, there is a sequence of second-order codes $(C_{n})_{n\in \N}$ such that $C_{n}$ equals $F_{n}$.  
Now, \emph{dense sub-sequences} are another second-order representation of closed sets, also called \emph{separably closed} sets.
This notion is equivalent to RM-closed for compact spaces in $\ACA_{0}$, by \cite{browner2}*{Theorem 3.3} and the same for sequences of such sets.
Hence, each $C_{n}$ is separably closed, which yields a double sequence $(x_{n,m})_{n,m\in \N}$ such that $x\in F_{n}\asa (\forall k\in \N)(\exists m\in \N)(|x-x_{n,m}|<\frac{1}{2^{k}})$.  
The definition of the closure of ${A}$ is now straightforward, as follows:  
\be\label{fooball}\textstyle
x\in \overline{A}\asa (\forall k\in \N)(\exists n,m\in \N)(|x-x_{n,m}|<\frac{1}{2^{k}}).  
\ee
Since the right-hand side of \eqref{fooball} is arithmetical, we are done. 
\end{proof}
Next, we have the following theorem expressing that the general existence of the closure and interior is much stronger than $\ATR_{0}$.
\begin{thm}[$\ACAo$]\label{tenar}
The following are equivalent.
\begin{itemize}
\item The principle $\BOOT$.
\item For any $X\subset \R$, the set $\INT(X)$ exists. 
\item For any $X\subset \R$, the set $\INT(X)$ exists and has an RM-code. 
\item For any $X\subset \R$, the set $\overline{X}$ exists. 
\item For any $X\subset \R$, the set $\overline{X}$ exists and has an RM-code. 
\end{itemize}
\end{thm}
\begin{proof}
To derive $\BOOT$, it suffices to obtain $\RANGE$, defined as follows:
\[
(\forall G:\N^{\N}\di \N)(\exists Z\subset \N)(\forall n \in \N)\big[n\in Z\asa (\exists f\in 2^{\N})(G(f)=n)  ].
\]
The equivalence between $\BOOT$ and $\RANGE$ is straightforward and may be found in \cite{samHARD, samBOOK}. 
Now fix $G:\N^{\N}\di \N$ and let $X\subset \R$ be the set of all $x\in \R$ such that
\[
 (\exists n\in \N, q\in \Q)\big(n+1<x-q<n+2 \wedge  G(\eta(x-n-1-q))=n\big).
\]
where $\eta$ is as in the proof of Theorem \ref{moare}.  
As above, we may ignore elements of Cantor space of the form $\sigma*11\dots$ for $\sigma \in 2^{<\N}$.     
We then have, for all $n\in \N$, that 
\be\label{woffo}
(\exists f\in 2^{\N})(G(f)=n)\asa (\exists r\in [n+1,n+2]\cap \Q)( r\in \overline{X} ),
\ee
which yields the range of $G$ as the right-hand side of \eqref{woffo} is arithmetical.

\smallskip

To define the interior of an arbitrary set $X\subset \R$, use $\BOOT$ to find $X_{0}\subset (\Q\times \N)$ such that $(q, N)\in X_{0}\asa [ B(q, \frac{1}{2^{N}})\subset X]$.  
Now observe that $\INT(X)=\cup_{(q, N)\in X_{0}}B(q, \frac{1}{2^{N}})$, which follows by definition.  Note that the latter union is an RM-code, i.e.\ all other items now follow.     
\end{proof}
Finally, we show that the fourth-order closure and interior operators are equivalent to Kleene's quantifier $(\exists^{3})$.
\begin{cor}[$\ACAo$]\label{fefequo}
The following are equivalent. 
\begin{itemize}
\item Kleene's quantifier $(\exists^{3})$. 
\item There exists $\mathfrak{I}:(\R\di \R)\di (\R\di \R)$ such that for any $X\subset \R$, $\mathfrak{I}(X)$ is the interior of $X$.  
\item There exists $\mathfrak{C}:(\R\di \R)\di (\R\di \R)$ such that for any $X\subset \R$, $\mathfrak{C}(X)$ is the closure of $X$.  
\end{itemize}
\end{cor}
\begin{proof}
That $(\exists^{3})$ proves the other items follows by the definition of interior and closure.  
For the other implications, fix $Y^{2}$ and use $(\exists^{2})$ to define $X_{Y}$ as the set of all $x\in \R$ with $(\exists q\in \Q)\big( Y(\eta(|x|-\lfloor |x|\rceil+ q\big)=0)$.  
Clearly, we have 
\[
(\exists f\in 2^{\N})(Y(f)=0)\asa (\exists r\in \Q )(r\in \overline{X_{Y}}),
\]
where we note that the right-hand side is arithmetical if the closure is given as in the final item.  
Since elements of Baire space can be coded as elements of Cantor space using $(\exists^{2})$, we obtain $(\exists^{3})$ as required. 
\end{proof}
We finish this section with a remark on how the above results yield equivalences for the Big Five, as explored in detail in \cite{samBOOK}. 
\begin{rem}\label{tiritomba}\rm
As in \cites{samBOOK, samtar}, equivalences involving e.g.\ $\open$ or $\BOOT$ can be `pushed down' to equivalences involving $\ATR_{0}$ or $\FIVE$ if we impose certain fairly natural restrictions 
on the function classes at hand.  One of these restrictions is to consider only \emph{effectively Baire 2} functions, which are given by the double limit of a double sequence of continuous functions.  Note that a Baire 2 function 
is `just' the limit of a sequence of Baire 1 functions, i.e.\ no double sequence is given.   

\smallskip

Other such `effective' restrictions are based on inverse images.  
For instance, a function $f:\R\di \R$ is measurable if the inverse image $f^{-1}(O)$ of any open set $O\subset \R$ is the union of an $\bf F_{\sigma}$-set and a set of measure zero. 
Then `effectively measurable' essentially means that all sets are given by codes.   Following \cites{samBOOK, samtar}, the proofs of Theorems \ref{moare} and \ref{tenar} yield equivalences for $\ATR_{0}$ and $\FIVE$, working over at least $\ACAo$, if we restrict to sets with an effectively Baire 2 characteristic function (or the other aforementioned effective notions).   
\end{rem}
In conclusion, we have obtained a number of equivalences involving the closure, interior, and boundary of certain classes of sets.  The basic case of open sets 
goes through assuming arithmetical comprehension, while the case of $\bf F_{\sigma}$-sets already yields $\ATR_{0}$ and the general case yields full second-order arithmetic in the guise of Feferman's projection principle and Kleene's quantifier $(\exists^{3})$.

\section{Sets and their representations}
\subsection{The $D$-representation of open sets}
We introduce the $D$-representation of open sets in Definition~\ref{charkar} and obtain equivalences in Theorem \ref{trefor} involving basic properties of $D$-open sets  and the enumeration principle $\cocode_{0}$.  The latter expresses that countable sets of reals can be enumerated, implies $\ATR_{0}$ when combined with $\ACAo$, and already boasts many equivalences from Fourier analysis and other fields (\cite{dagsamXI, samBIG, samBOOK, samBIG3}).   

\smallskip

\smallskip
\noindent
First of all, we consider the following representation of open sets which appears to be a slight generalisation of the RM-definition.   
\bdefi\label{charkar} Let $A\subset \R$ be a set.  We say that:
\begin{itemize}
\item $A$ is \emph{$D$-open} if there is a sequence $(J_{n})_{n\in \N}$ such that $O=\cup_{n\in \N}J_{n}$ and each $J_{n}$ is either an open interval or the union of two open intervals,
\item $A$ is \emph{$D$-closed} if the complement is $D$-open, 
\item $A$ is ${\bf F}^{D}_{\sigma}$ if it is the countable union of $D$-closed sets. 
\end{itemize}
\edefi
\noindent
Note that we generally cannot check which case holds in the definition of $D$-open sets as quantifiers over $\R$ are involved.
Nonetheless, $D$-closed sets have an RM-code if they are also perfect, which we found surprising.  
\begin{thm}[$\ACAo$]\label{zefrre} 
A perfect $D$-closed set of reals is RM-closed.  
\end{thm}
\begin{proof}
Suppose $C\subset \R$ is $D$-closed and perfect.  
Let $O:=\R\setminus C$ satisfy $O=\cup_{n\in \N}J_{n}$ where the latter is a $D$-representation.
Consider $\tilde{O}:= \cup_{n\in \N}  I_{n}$ where $I_{n}=\INT(\overline{J_{n}})$ is defined using Theorem \ref{talkingaboud}.
We now show that $O=\tilde{O}$. 
Indeed, we trivially have $O\subseteq \tilde{O}$ and suppose $x_{0}\in \tilde{O}\setminus O$.
If $x_{0}\in I_{n_{0}}$, then $x_{0}\not \in J_{n_{0}}$ and $x_{0}\in C$, i.e.\ $x_{0}$ is an isolated point of $C$, contradiction.
\end{proof}
By Theorem \ref{zefrre}, $\ATR_{0}$ now follows from the \emph{perfect set theorem} as follows: 
\begin{center}
\emph{any non-enumerable $D$-closed set has a $D$-closed and perfect subset}, 
\end{center}
as the latter for RM-codes implies $\ATR_{0}$ by \cite{simpson2}*{V.5.5}.  
We shall in fact obtain an equivalence to $\cocode_{0}$ in Theorem \ref{trefor}.

\smallskip

Secondly, we establish basic properties of $D$-open sets using only arithmetical comprehension as in $\ACAo$.  
Using Definition \ref{charkar}, we can now formulate the following $D$-versions of continuity, semi-continuity\footnote{We use `usco' to abbreviate `upper semi-continuous'.}, and Baire 1.
Indeed, dropping `$D$-' in the following definition yields the usual definitions (in some equivalent form).  
\bdefi For $f:\R\di \R$, we say that:
\begin{itemize}
\item $f$ is $D$-continuous if $f^{-1}(O)$ is $D$-open for $D$-open $O\subset \R$,
\item $f$ is $D$-usco if the level sets $\{ x\in \R: f(x)<y\}$ are $D$-open for all $y\in \R$, 
\item $f$ is $D$-Baire 1 if $f^{-1}(O)$ is ${\bf F}_{\sigma}^{D}$ for any $D$-open $O\subset \R$,
\item $f$ is Baire 1 if it is the pointwise limit of a sequence of continuous functions,
\item $f$ is Borel measurable of class $1$ if $f^{-1}(O)$ is ${\bf F}_{\sigma}$ for any open $O\subset \R$.
\end{itemize}
\edefi
Many function classes are defined via open or closed sets, i.e.\ we could generalise Theorem \ref{trefor} along these lines.  
In the case of continuity, there is no real difference.   
\begin{thm}[$\ACAo$]
A function on the reals is $D$-continuous iff it is continuous
\end{thm}
\begin{proof}
Fix $f:\R\di \R$ and suppose it is $D$-continuous
Then $O:=B(f(x), \frac{1}{2^{k}})$ is $D$-open and there is $N\in \N$ with  $B(x, \frac{1}{2^{N}})\subset  f^{-1}(O)$ as the latter is $D$-open by assumption.
Hence, $f$ is (epsilon-delta) continuous, as required. 

\smallskip

Now suppose $f$ is continuous and let $O=\cup_{n\in \N}J_{n}$ be $D$-open.  
Define $I_{n}:= \INT(\overline{J_{n}})$ using Theorem \ref{talkingaboud} and note that $E_{n}:=I_{n}\setminus J_{n}$ is at most a singleton. 
Since $f$ is continuous, it is uniformly continuous on each $[-m, m]$ (see \cite{dagsamXIV}*{\S2}), say with modulus of uniform continuity $g_{m}:\R\di \R$ such that $|x-y|<g_{m}(\eps)$ implies $|f(x)-f(y)|<\eps$ for any $x, y\in [-m,m]$. 
Define $N_{n, q, m}$ as $B(q, g_{m}(  d(q,\R \setminus I_{n}) ))$ and observe that $f^{-1}(I_{n})$ equals $\cup_{m\in \N}\cup_{q\in \Q\cap [-m, m], f(q)\in I_{n}}N_{n, q, m}$.  
Now define $M_{n, q, m}:=N_{n, q, m}\setminus E_{n} $ and observe that this is either an open interval or the union of two open intervals.  
By definition, $f^{-1}(O)$ equals the countable union  $\cup_{m,n\in \N}\cup_{q\in \Q\cap [-m, m], f(q)\in I_{n}}M_{n, q, m}$, which is as required after some straightforward modification using $(\exists^{2})$. 
\end{proof}
The connection between closed sets and usco functions is well-known:  a set $C\subset\R$ is closed iff $\mathbb{1}_{C}$ is usco.  
This also holds for the $D$-representation. 
\begin{thm} ~
\begin{itemize}
\item Over $\ACAo$, the set $C\subset \R$ is $D$-closed iff $\mathbb{1}_{C}$ is $D$-usco. 
\item Over $\RCAo+\WKL_{0}$, the set $C\subset [0,1]$ is $D$-closed iff there exist a $D$-usco representation function $F_{C}:[0,1]\di \R$ for $C$. 
\end{itemize}
\end{thm}
\begin{proof}
For the first item, the level set $\{ x\in \R: \mathbb{1}_{C}(x)<y \}$ is either $\emptyset$, $C$, or $\R$ depending on whether $y<1$, $0\leq y \leq 1$, or $y>1$.  
Since $\emptyset, \R$ are even RM-closed, the first item follows.  
For the second item, fix $C\subset [0,1]$ and invoke the law of excluded middle as in $(\exists^{2})\vee \neg(\exists^{2})$.  
In case $(\exists^{2})$, define the characteristic function $\mathbb{1}_{C}$ and note that it is a $D$-usco representation function iff $C$ is $D$-closed. 
In case $\neg (\exists^{2})$, all $\R\di \R$-functions are continuous by \cite{kohlenbach2}*{Prop.\ 3.14}.  Hence, any representation function $F_O$ for $O=\R\setminus C$ is continuous.
By \cite{dagsamXIV}*{\S2}, the function $F_{O}$ has an RM-code assuming $\WKL_{0}$, implying that $O$ is RM-open by \cite{simpson2}*{II.7.1}.  
\end{proof}
We only need $(\exists^{2})$ in the previous theorem to guarantee that the characteristic function exists.  An general approach that works over $\RCA_{0}$ is explored in \cite{samBOOK}.

\smallskip

Next, we now establish equivalences for other basic properties of $D$-open sets. 
We note that item \eqref{waja} boasts many equivalences involving Fourier analysis already (\cites{samBIG2, samBIG3, samBOOK}). 
We could have adapted some the proofs from \cite{samcie26, samBOOK} but choose to provide direct proofs.
\begin{thm}[$\ACAo+\QFAC^{0,1}$]\label{trefor}
The following are equivalent. 
\begin{enumerate}
\renewcommand{\theenumi}{\alph{enumi}}
\item The enumeration principle $\cocode_{0}$: for any countable set $A\subset [0,1]$, there is a sequence of reals that includes all elements of $A$.\label{waja} 
\item The principle $\open_{D}$: for every $D$-open $O\subset \R$, there is a sequence of open intervals $(I_{n})_{n\in \N}$ such that $O=\cup_{n\in \N}I_{n}$.\label{wajb} 
\item For every $D$-closed $C\subset \R$, the distance function $\lambda x.d(x, C)$ exists.\label{wajb2} 
\item For every $D$-closed $C\subset \R$, the supremum $\sup (C\cap [p, q])$ exists as a sequence with $p, q\in \Q$.\label{wajb3} 
\item The perfect set theorem for $D$-closed sets:  \emph{any non-enumerable $D$-closed set has a $D$-closed and perfect subset}.\label{peffie}
\item For any set $X\subset \R$ that is ${\bf F}_{\sigma}^{D}$, the closure $\overline{X}$ exists.\label{salmuw}
\item The supremum principle for $D$-usco functions: for $D$-usco $f:[0,1]\di \R$, the reals $\sup_{x\in [p, q]}f(x)$ exist as a sequence over $\Q\cap [0,1]$.\label{wajc2}  
\item  Any $D$-usco function on the unit interval is Baire 1.\label{wajc4}  
\item The maximum principle: for $D$-usco functions $f:\R\di \R$, there is a sequence $(x_{p, q})_{p, q\in \Q }$ with $(\forall y\in [p, q])(f(y)\leq f(x_{p, q}))$ for any $p, q\in \Q$.\label{wajc3}  
\item The supremum principle for continuous functions:  for $D$-closed $C\subset [0,1]$ and $f:[0,1]\di \R$ continuous on $C$, the reals $\sup_{x\in C\cap [p,q]}f(x)$ exist as a sequence over $\Q\cap [0,1]$. \label{wajc}  
\item The supremum principle for $\Gamma$-functions on $D$-closed sets in the unit interval, where $\Gamma$ is between the classes of continuous and $D$-usco functions.\label{wartaal} 
\item \(Urysohn\) For $D$-closed disjoint $C_{0}, C_{1}\subseteq \R$, there is continuous $g:\R\di [0,1]$ such that $x\in C_{i}\asa g(x)=i$ for any $x\in \R$ and $i\in \{0,1\}$.\label{wajd}
\item The combination of the following.  \label{wajd2}
\begin{itemize}
\item[(m.1)] \(Tietze\) For $D$-closed $C\subset \R$ and $f:\R\di [0,1]$ continuous on $C$, there is continuous $g:\R\di \R$ such that $f=g$ on $C$. 
\item[(m.2)] The principle $\open_{D}^{\dagger}$: every $D$-open set is ${\bf F}_{\sigma}^{D}$. 
\end{itemize}
\item The combination of the following.\label{helavista2}
\begin{itemize}
\item[(n.1)] The supremum principle for bounded $D$-Baire 1 functions on $[0,1]$.
\item[(n.2)] The principle $\open^{\dagger}_{D}$:  every open set of reals is ${\bf F}_{\sigma}^{D}$.  
\end{itemize}
\item The combination of the following.\label{helavista3}
\begin{itemize}
\item[(o.1)] \(Lebesgue, \cite{veellebf}\) Any $D$-Baire 1 function on $[0,1]$ is Baire 1.
\item[(o.2)] The principle $\open^{\dagger}_{D}$:  every open set of reals is ${\bf F}_{\sigma}^{D}$.  
\end{itemize}
\item  \(Lindel\"of\) For $D$-closed $C\subset \R$ and continuous $\Psi: \R\di \R$ that is non-negative on $C$, there is a sequence $(x_{n})_{n\in \N}$ in $C$ such that $C\subset \cup_{n\in \N}B(x_{n}, \Psi(x_{n}))$.\label{blindeloefz}  
\end{enumerate}
We only use $\QFAC^{0,1}$ to establish $\eqref{waja}\di \eqref{wajc3}, \eqref{wajd2}$. 
\end{thm}
\begin{proof}
First of all, $\cocode_{0}$ is equivalent to the following items (\cite{dagsamX, dagsamXI}).
\begin{itemize}
\item The statement $\range_{0}$ which expresses that for a countable set $A\subset [0,1]$ and $Y:[0,1]\di \N$ injective on $A$, the range of the latter on the former exists, i.e.\ there is $X\subset \N$ with $n\in X\asa (\exists x\in A)(Y(x)=n)$ for all $n\in \N$ 
\item For any arithmetical $\varphi$ such that $(\forall n\in \N)(\exists \textup{ at most one } f\in 2^{\N})\varphi(n, f)$, there is $Z\subset \N$ such that $n\in Z\asa (\exists f\in 2^{\N})\varphi(n, f)$, for all $n\in \N$. 
\end{itemize}
Hence, by the proof of Theorem \ref{moare}, item \eqref{salmuw} implies $\cocode_{0}$.    
Moreover, the former proof is readily modified to show that $\open_{D}$ implies item \eqref{salmuw}.  
The equivalence $\cocode_{0}\asa \open_{D}$ is proved below.  One readily shows that $\open_{D}$ follows from its restriction to the unit interval.  

\smallskip

Secondly, assume item \eqref{waja} and fix an open $O\subset \R$ with $D$-representation $(J_{n})_{n\in \N}$
By Theorem \ref{talkingaboud}, the closure $\overline{J_{n}}$ exists and $\overline{J_{n}}\setminus J_{n}$ is finite (with at most four elements).  
Hence, we can enumerate the union $\cup_{n\in \N}\big(\overline{J_{n}}\setminus J_{n} \big)$ using $\cocode_{0}$, which immediately yields an RM-code for $O$. 
Indeed, we can split $J_{n}$ in its constituent intervals given the end-points of the latter.  Hence, item \eqref{wajb} follows. 

\smallskip

Thirdly, assume item \eqref{wajb}, fix $A\subset [0,1]$ and $Y:\R\di \N$ with the latter being injective on the former.  
Now define the closed set $B\subset \R$ as follows: $x\in B$ if
\be\label{poil}
(\exists n\in \N)[ x \in [2n, 2n+1]\wedge x-2n \in A \wedge Y(x-2n)=n  ].  
\ee
Hence, each unit interval contains at most one element of $B$, i.e.\ the $D$-representation of $O:=\R\setminus B $ is readily defined.  
Suppose $O=\cup_{n\in \N}(a_{n}, b_{n})$ and note that 
\begin{align}
(\exists x\in A)(Y(x)=n)&\asa \cup_{n\in \N}(a_{n}, b_{n}) \textup{ does not cover $[2n, 2n+1]$} \notag \\
&\asa (\forall m\in \N)(\cup_{n\leq m}(a_{n}, b_{n}) \textup{ does not cover $[2n, 2n+1]$}.\label{woeif}
\end{align}
Note that the formula in \eqref{woeif} arithmetical, i.e.\ we can define the range of $Y$ on $A$, and hence $\cocode_{0}$ follows by the previous paragraph.

\smallskip

Fourth, item \eqref{wajb2} follows from item \eqref{wajb} as RM-closed sets have a distance function assuming $\ACA_{0}$ (\cite{withgusto, browner}).  
To obtain $\cocode_{0}$, consider the $D$-closed set $B$ defined by \eqref{poil}.  Essentially by definition, we have that
\[\textstyle
(\exists x\in A)(Y(x)=n)\asa d( 2n+\frac{1}{2}, B)<\frac{1}{2}. 
\]
Since the right-hand side is arithmetical, we obtain $\range_{0}$ as required.  
A similar proof yields the equivalence between $\cocode_{0}$ and item \eqref{wajb3}, noting that $\ACA_{0}$ proves the existence of the 
supremum for RM-closed sets of reals (\cite{simpson2}*{IV.2.11}).  

\smallskip
  
Fifth, for \eqref{waja} $\di$ \eqref{peffie}, $\ACAo+\cocode_{0}$ proves $\ATR_{0}$ (\cite{dagsamXI}) while the latter is equivalent to the second-order perfect set theorem (\cite{simpson2}*{I.11.5}).  
To prove item \eqref{peffie}, use $\open_{D}$ to convert a given non-enumerable $D$-closed set to an RM-code, apply the second-order perfect set theorem, and note that an RM-closed set is also $D$-closed (by definition).  
For the reversal,  fix $A\subset [0,1]$ and $Y:[0,1]\di \R$ with the latter injective on the former.
Consider the set $B$ from \eqref{poil}, which is $D$-closed but clearly has no perfect subset as every point is isolated.  By the contraposition of item \eqref{peffie}, $B$ can be enumerated, as required for $\cocode_{0}$. 

\smallskip

Sixth, to prove the supremum principle as in item \eqref{wajc2} from $\open_{D}$, fix $D$-usco $f:[0,1]\di \R$ and consider the level sets $L_{q}:=\{x\in [0,1]:f(x) \geq q\}$ for $q\in \Q$, which are $D$-closed by assumption.  
Use $\open_{D}$ to obtain a sequence of open intervals $(I_{n, q})_{n\in \N, q\in \Q}$ such that $[0,1]\setminus L_{q}=\cup_{n\in \N}I_{n, q}$.  
To show that $f$ is bounded on $[0,1]$, suppose $(\forall n\in \Q)(\exists x\in [0,1])( x\in L_{n})$.  Since the latter sets are RM-closed, there is a sequence $(x_{n})_{n\in \N}$ with $x_{n}\in L_{n}$ for each $n\in \N$ by \cite{simpson2}*{IV.1.8}.   As $\ACA_{0}$ proves the sequential compactness of $[0,1]$ (\cite{simpson2}*{III.2}), there is a convergent sub-sequence, say with limit $y\in [0,1]$.  Clearly, $f$ is not $D$-usco at $y$, a contradiction, and $f$ is bounded on $[0,1]$.  Next, we also have have the following  
\be\textstyle\label{texzow}
L_{q}=\emptyset \asa [0,1]\subseteq \cup_{n\in \N}I_{n ,q}\asa (\exists n_{0}\in \N)([0,1] \subseteq \cup_{n\leq n_{0}} I_{n, q}),
\ee
where the final equivalence is due to the countable Heine-Borel theorem, provable in $\WKL_{0}$ (\cite{simpson2}*{IV.1}). 
However, the right-hand side of \eqref{texzow} is decidable using $(\exists^{2})$.  The usual interval-halving technique (also using $(\exists^{2})$) now readily yields $\sup_{x\in [0,1]}f(x)$, and the same proof goes through for intervals $[p, q]$ for $p, q\in \Q$.   

\smallskip

For item \eqref{wajc4}, the supremum principle for Baire 1 functions is provable in $\WKL_{0}$ (\cite{dagsamXIV}*{\S2}), i.e.\ the former item implies item \eqref{wajc2}.  
To obtain the reversal, fix $D$-usco $f:[0,1]\di \R$ define the continuous function $f_{n}:[0,1]\di \R$ as follows: 
\be\label{zaikin}\textstyle
f_{n}(x):= \sup_{y\in [0,1]}  (f(y)- n\cdot |x-y| ),
\ee
where the supremum is defined by first restricting $x$ to the rationals and then taking limits.  
Essentially, the sequence $(f_{n}(q))_{q\in \Q\cap [0,1]}$ yields a code for a continuous function. 
Clearly, $f$ is the pointwise limit of $(f_{n})_{n\in \N}$ as required for item~\eqref{wajc4}.

\smallskip

To obtain $\open_{D}$ from item \eqref{wajc2}, let $O\subset [0,1]$ be $D$-open and consider $f(x)=\mathbb{1}_{C}(x)$ where $C$ is $[0,1]\setminus O$.
Clearly, $f$ is $D$-usco and note that $[p, q]\subseteq O\asa \sup_{x\in [p, q]}f(x)=_{\R}0$ for all $p, q\in \Q\cap [0,1]$.   
Since the right-hand side is decidable using $(\exists^{2})$, an RM-code of $O$ is readily defined and $\open_{D}$ follows.  Clearly, item \eqref{wajc3} implies item \eqref{wajc2} and to obtain the former,   
observe that the definition of supremum implies $(\forall k\in \N)(\exists y\in [0,1])(f(y)\geq \sup_{x\in [0,1]}f(x)-\frac{1}{2^{k}} )$.  Apply $\QFAC^{0,1}$ and let $(y_{n})_{n\in \N}$ be the resulting sequence. 
As $\ACAo$ proves the sequential compactness of $[0,1]$, the latter sequence has a convergent sub-sequence, say with limit $z\in [0,1]$.  Then $f(y)\leq f(z)$ for all $y\in[0,1]$ as required, and the general case readily follows.  

\smallskip

Seventh, item \eqref{wajc} immediately implies item \eqref{wajb3}.  For the reversal, fix $f, C$ as in the latter and assume $\open_{D}$.  Now recall that RM-closed sets are separably closed in $\ACA_{0}$ by \cite{browner2}*{Theorem 3.3}.  Hence, there is a sequence $(x_{n})_{n\in \N}$ such that $x\in C\asa (\forall k\in \N)(\exists m\in \N)(|x-x_{m}|<\frac{1}{2^{k}})$.  
The continuity of $f$ on $C$ and a case distinction (based on whether a real is an isolated point of $C$ or not) then yields
\[
(\exists x\in C)(f(x)>q)\asa (\exists n\in \N)(f(x_{n})>q)
\]
which readily yields the required supremum.  

\smallskip
 
Eighth, to prove item \eqref{wartaal} for $\Gamma$ equal to the $D$-usco functions, observe that the union of two $D$-open sets is again $D$-open.
Hence, $C\cap L_{q}$ is $D$-closed if $C$ is $D$-closed and $L_{q}:=\{x\in [0,1]:f(x) \geq q\}$ is the level set of the D-usco function $f:[0,1]\di \R$. The proof is now analogous to that of item \eqref{wajc2}.
  
\smallskip 

Ninth, the Urysohn lemma for RM-codes is provable in $\RCA_{0}$ (\cite{simpson2}*{II.7}).  Hence, item \eqref{wajd} is immediate from $\open_{D}$.   
To obtain $\cocode_{0}$ from the Urysohn lemma as in item \eqref{wajd}, fix countable $A\subset [0,1]$ and consider the associated $D$-closed set $B$ defined from \eqref{poil}.
Put $C_{1}=B$ and $C_{0}=\emptyset$ and let $g:\R\di \R$ be the continuous function satisfying $g(x)=i\asa x\in C_{i}$ for $i=0,1$.  
By \cite{kohlenbach2}*{\S3}, $(\exists^{2})$ suffices to define the supremum as in the right-hand side of \eqref{zeazy} as follows:
\be\textstyle\label{zeazy}
(\exists x\in A)(Y(x)=n)\asa \notag [B\cap [2n, 2n+1]\ne \emptyset] \asa[ \sup_{y\in [2n, 2n+1]}g(y)=_{\R}1].
\ee
Note that the final equivalence of \eqref{zeazy} makes use the fact that the maximum principle for continuous functions is provable in $\WKL_{0}$, 
for coded and non-coded functions (see \cite{simpson2}*{IV.2} and \cite{dagsamXIV}*{\S2}).  Clearly, the right-hand side of \eqref{zeazy} is arithmetical, and $\range_{0}$ follows. 
Alternatively, one could modify the proof of \cite{samcie26}*{Theorem 17} to $D$-open sets, but the above proof seems more elegant. 
%

\smallskip

Next, the equivalence involving the Tietze extension theorem is based on the proofs in \cite{samcie26, samBOOK} and we only provide a sketch.  
Now, to prove (m.1), one can use Hausdorff's construction from \cite{huisdorp}, where $f$ is continuous on the closed set $C$:
\be\label{shartcot}
g(x):=
\begin{cases}
f(x)  & x\in C \\
\inf_{y\in C}\big( f(y)+\frac{|x-y|}{d(x, C)}-1\big) & x\not\in C
\end{cases}.
\ee
The infimum and distance function in \eqref{shartcot} are well-defined assuming $\open_{D}$.  
Alternatively, the latter provides an RM-code for the closed set $C$, which yields a dense sub-sequence $(x_{n})_{n\in \N}$ of $C$ by \cite{browner2}*{Theorem 3.3} .
The function $f$ on $C$ can be represented by $(f(x_{n}))_{n\in \N}$ which readily yields the necessary modulus of continuity.  The second-order proof from \cite{simpson2}*{II.7} now goes through.   
To prove $\open^{\dagger}_{D}$ from $\open_{D}$, recall that the distance function $d(x, C)$ exists for RM-closed sets $C\subset \R$, say in $\ACAo$ (\cite{withgusto, browner2}).  
For $D$-open $O\subset \R$, observe that $O=\cup_{n\in \N} F_{n}$ for the closed set $F_{n}:=\{x\in \R : d(x, \R\setminus O)\geq \frac{1}{2^{k}} \}$.  
By \cite{simpson2}*{II.7.1}, $F_{n}$ is RM-closed and hence $D$-closed, as required.  

\smallskip

To show that $\open_{D}$ follows from item \eqref{wajd2}, let $O\subset [0,1]$ be $D$-open and suppose $O=\cup_{n\in \N}F_{n}$ where each $F_{n}$ is $D$-closed using $\open^{\dagger}_{D}$.  
Define $F$ as $ [0,1]\setminus O$ and define $f_{n}(x)$ as $0$ if $x\in F$ and $\frac{1}{2^{n}}$ if $x\in F_{n}$.  Then $f_{n}:[0,1]\di \R$ is continuous on $F\cup F_{n}$ and let $g_{n}$ be the continuous extension\footnote{Technically, one needs to shift $f_{n}$ and $F\cup F_{n}$ to $[n+1, n+2]$ to observe that the Tietze theorem for $\R$ implies a sequential version of the Tietze theorem on $[0,1]$.} of $f_{n}$ to $[0,1]$ provided by the Tietze extension theorem.    
Use the $M$-test of Weierstrass (provable in $\RCA_{0}$ by \cite{simpson2}*{II.6.5}) to show that $g(x):=\sum_{n=0}^{\infty}g_{n}(x)$ exists and is continuous, rescaling $g_{n}$ if necessary.  
Then $g$ satisfies $x\in O\asa g(x)>0$ and one readily defines a code for $O$ using $g$.  

\smallskip

Next, to prove item \eqref{helavista2} from $\open_{D}$, fix a $D$-Baire 1 function $f:[0,1]\di [0,1]$.  
Now consider the sequence of ${\bf F}^{D}_{\sigma}$-sets $f^{-1}(E_{q})$ where $E_{q}:=(q, +\infty)$ and $q\in \Q$.  
Let $(I_{n, m, q})_{n, m\in \N, q\in \Q}$ be a sequence of intervals in $[0,1]$ provided by $\open_{D}$ such that for all $q\in \Q$, we have $f^{-1}(E_{q})=\cup_{n\in \N}  C_{n, q}$ with $C_{n, q}$ the complement of the RM-open set $O_{n, q}:= \cup_{m\in \N}I_{n,m, q}$ in the unit interval.   Then we have
\begin{align*}
(\forall x\in [0,1])(f(x)\leq q)\asa f^{-1}(E_{q})=\emptyset  &\asa (\forall n\in \N)(C_{n, q}=\emptyset)\\
&\asa (\forall n\in \N)([0,1]\subset \cup_{m\in \N}I_{n,m,q} )\notag\\
&\asa (\forall n\in \N)(\exists m_{0}\in \N)([0,1]\subset \cup_{m\leq m_{0}}I_{n,m,q} ),\notag
\end{align*}
where we used the (second-order) Heine-Borel theorem in the final step.
Since the final formula in the previous equation is (equivalent to) arithmetical, the usual interval-halving technique using $(\exists^{2})$ yields $\sup_{x\in [0,1]}f(x)$; the modification to intervals with rational end-points is straightforward.  
For the reversal, assume item~\eqref{helavista2}, let $C\subset \R$ be $D$-closed, and note that $\mathbb{1}_{C}$ is $D$-Baire 1.  
Indeed, a case distinction shows that for any $D$-open $O\subset \R$, the set $\mathbb{1}_{C}^{-1}(O)$ is either $\emptyset, \R, C$, or $\R\setminus C$.  
Since $\open^{\dagger}_{D}$ is given, each of these is ${\bf F}_{\sigma}^{D}$.  As above, the supremum provided by item \eqref{helavista2} yields an RM-code for $C$ as $[p, q]\subset C^{c}\asa [0=_{\R} \sup_{x\in [p, q]}\mathbb{1}_{C}(x)]$ for any $p, q\in \Q\cap [0,1]$, where the right-hand side is decidable in $\ACAo$.  
%
%
%
%

\smallskip

Next, for item \eqref{helavista3}, it suffices to prove (o.1) from $\open_{D}$, which follows by modifying the proof of \cite{overderooie}*{Theorem 11.12} for RM-codes where necessary.  
The latter proof consists of the following steps.  
\begin{enumerate}
\renewcommand{\theenumi}{\roman{enumi}}
\item The class of Baire 1 functions is closed under \emph{uniform} limits.\label{jitem1}
\item The characteristic functions of ambivalent, i.e.\ ${\bf F}_{\sigma}\cap {\bf G}_{\delta}$, sets are Baire 1.\label{jitem3}  
\item A Borel measurable function of class $1$ is the uniform limit of finite sums of characteristic functions of ambivalent sets.\label{jitem2}  
\end{enumerate}
The proof of item \eqref{jitem1} in \cite{overderooie}*{Theorem 11.7} is straightforward and only really makes use of the $M$-test of Weierstrass, which is provable in $\ACAo$ via the usual proof.  

\smallskip

To prove item \eqref{jitem3} for RM-codes, let $X\subset \R$ be ambivalent, i.e.\ there are increasing sequences $(C_{n})_{n\in \N}$ and $(P_{m})_{m\in \N}$ of closed sets with $X=\cup_{n\in \N}C_{n}$ and $\R\setminus X=\cup_{m\in \N}P_{m}$.  Define the continuous function $f_{n}(x):=\frac{d(x, P_{n})}{d(x, P_{n})+d(x, C_{n})}$, noting that the distance function exists for RM-codes in $\ACA_{0}$ (\cite{withgusto, browner}). 
Clearly, $\mathbb{1}_{X}(x)$ equals $\lim_{n\di \infty}f_{n}(x)$ for all $x\in \R$, as required. 

\smallskip

To prove item \eqref{jitem2}, let $f:\R\di [0,1]$ be Borel measurable of class $1$ and consider the following sets, which are $\bf F_{\sigma}$ by definition and where $i\leq  2^{k}$:
\[\textstyle
E_{i}:= \{x\in \R:  \frac{i-1}{2^{k}}\leq f(x)\leq \frac{i+1}{2^{k}} \}.
\]
Clearly, $\cup_{i\leq k}E_{i}$ covers $\R$ but the sets are not pairwise disjoint.      
Now suppose there are RM-closed $(C_{n, i})_{n\in \N}$ with $E_{i}=\cup_{n\in \N}C_{n, i}$.  
Let $\pi$ be a standard pairing function and define the RM-closed set $D_{n, i}:= C_{n, i}\setminus \bigcup_{\textup{all $(m, j)$ with }\pi(m, j)<\pi(n, i)} C_{m, j} $. 
Then $F_{i}=\cup_{n\in \N}D_{n, i}$ satisfies $F_{i}\subset E_{i}$ for all $i\leq 2^{k}$ and the $F_{i}$ are pairwise disjoint by definition.  
Moreover, $\cup_{i\leq 2^{k}}F_{i}$ still covers $\R$, i.e.\ the $F_{i}$ are also ambivalent.    
Now define $g_{k}(x):=\sum_{i\leq 2^{k+1}}\frac{i}{2^{k+1}}\mathbb{1}_{F_{i}}(x)$, which is Baire 1 by item \eqref{jitem3}. 
By definition, we have $|f(x)-g_{k}(x)|<\frac{1}{2^{k+1}}$ for all $x\in \R$, i.e.\ $f$ is the uniform limit of $(g_{k})_{k\in \N}$ and hence the former is Baire 1 by item \eqref{jitem1}.

\smallskip

Finally, item \eqref{blindeloefz} implies $\range_{0}$ by considering the $D$-closed set $B$ as defined using \eqref{poil}.  
Indeed, let $\Psi:\R\di \R$ be $0$ on $\N\cup \R^{-}$  and $\min(|x-n|, |x-(n+1)|)$ if $n<x< n+1$.  
Let $(x_{n})_{n\in\N}$ be the sequence provided by item \eqref{blindeloefz} and note that 
\be\label{zoltan}
(\exists x\in A)(Y(x)=n)\asa (\exists n\in \N)(x_{n}\in A\wedge Y(x_{n})=n).
\ee
Since the right-hand side of \eqref{zoltan} is arithmetical, $\range_{0}$ follows.  To prove item~\eqref{blindeloefz} from $\open_{D}$, observe that any dense sub-sequence of $C$ is as required for the Lindel\"of lemma as in item \eqref{blindeloefz}, as $\Psi$ is continuous.  
This sub-sequence exists for RM-closed sets in $\ACA_{0}$ by \cite{browner2}*{Theorem 3.3}.  
\end{proof}
Regarding item \eqref{helavista2}, the supremum principle for Baire 1 functions (with the usual definition) is equivalent to $\WKL_{0}$ (\cite{dagsamXIV}*{\S2}), a system much weaker than $\ATR_{0}$.  
Moreover, item \eqref{helavista2} 
seems to readily generalise to Baire $n$-functions.   
We are interested in exploring more equivalences involving $\open_{D}^{\dagger}$.  The latter is equivalent to (a version of the) the fact that $D$-usco functions are $D$-Baire 1.  

\smallskip

%
%
We finish this section with some conceptual remarks, including future work. 
%
\begin{remark}[Second-order equivalences]\label{beyatr}\rm
The observations from Remark \ref{tiritomba} of course apply to Theorem \ref{trefor}.  For instance, the equivalence between $\cocode_{0}$ and item \eqref{wajc2} of the latter can be `pushed down' to an equivalence between $\ATR_{0}$ and the supremum principle for $D$-usco functions that are also effectively Baire 2, i.e.\ given by a double limit of a double sequence of continuous functions.  
Other `effective' definitions can of course be used, as discussed in Remark \ref{tiritomba}
\end{remark}
\begin{rem}[Beyond $\ATR_{0}$]\label{beyatr2}\rm 
The system $\ACAo+\cocode_{0}$ is a conservative extension of $\ATR_{0}$ (\cite{samBOOK}), which follows by an easy modification of Hunter's conservation results from \cite{hunterphd}. 
Thus, the results in this paper exist at the level of $\ATR_{0}$ and it is a natural question if there are results at the level of the strongest Big Five system, $\FIVE$.  
While we deem this mostly `future work', we mention one example here, namely that the following centred statement is equivalent to $\cocode_{0}+\FIVE$.  
\begin{center}
\emph{The Cantor-Bendixson theorem for $D$-closed sets: any $D$-closed $C\subset \R$ is the disjoint union of a $D$-closed and perfect set $P$ and an enumerable set $S$.}
\end{center}
The crucial element in the previous statement is that the $S$ is enumerable (not just countable).  This constitutes `extra data' and it is a natural question, in the very spirit of RM, what
the status is of the `unrepresented' Cantor-Bendixson theorem, i.e.\ where the set $S$ is only assumed to be countable (=injection to $\N$).  We provide a (perhaps surprising) answer to this question in Section \ref{mummershines}, which 
also answers a question by Mummert at the 2022 RM-meeting in Paris.  
%
\end{rem}

\begin{remark}[Enumeration and representation]\rm
First of all, $\cocode_{0}$ was introduced in \cite{dagsamXI} while the slightly more general principle $\enum$ expresses: 
\begin{center}
\emph{if a set $A=\cup_{n\in \N}A_{n}$ where each $A_{n}\subset \R$ is finite, then $A$ can be enumerated.  }
\end{center}
Assuming rather weak principles in the base theory, these principles are equivalent.  Nonetheless, for the RM-study of the uncountability of the reals, 
the more general countability notion, i.e.\ `union over $\N$ of finite sets' rather than `injection to $\N$', is most fruitful/practical.  The reason is that this RM-study 
involves basic theorems of Fourier analysis and for functions of bounded variation $f:[0,1]\di \R$, the discontinuity set $D_{f}$ exists and is the union over $\N$ of finite sets, say over $\ACAo+\QFAC^{0,1}$.  
There however seems to be no way of obtaining an injection from $D_{f}$ to $\N$ in the latter system, and much stronger systems like $\Z_{2}^{\omega}+\QFAC^{0,1}$.  
  
\smallskip

Secondly, in light of the previous, one could generalise the $D$-representation to:  
\begin{center}
\emph{each $J_{n}$ from Definition \ref{charkar} equals a finite union of basic open intervals.}
\end{center}
We have chosen the more elementary option as we were interested in equivalences for $\cocode_{0}$: the latter yields $\ATR_{0}$ already when combined with $\ACAo$.   
A similar path was taken here and there in \cite{samBOOK} when the technical details became daunting.  
Nonetheless, we are interested in the most general version of the $D$-representation and its future development. 
For instance, the centred generalisation is convenient for treating finite intersections of $D$-open sets.  Nonetheless, the $D$-representation from Definition \ref{charkar} already gives rise to $\ATR_{0}$ and the same of course holds for any generalisation, following Theorem \ref{trefor}.  Also, one often needs extra induction as in $\SIND$ when dealing with finite unions.  
\end{remark}
\begin{remark}[Enumeration and representation II]\rm
The equivalence between $\ATR_{0}$ and the perfect set theorem as in item \eqref{peffie} is conceptually pleasing:  replacing the second-order representation by the $D$-representation, one obtains
a principle equivalent to $\cocode_{0}$, where the latter exists at the level of $\ATR_{0}$ assuming $\ACAo$.  There is however one cosmetic blemish: item \eqref{peffie} is formulated using \emph{non-enumerable} rather than \emph{uncountable}.  It is then a natural question what the strength is of the perfect set theorem formulated with `uncountable'.        
We provide an answer in Section~\ref{mummershines} that the reader may not entirely welcome:  without the `extra data' provided by enumerations, the perfect set theorem (and many similar statements about uncountable sets) does not imply $\ATR_{0}$. 
\end{remark}
\begin{remark}[Representations, measure and category]\label{meaca}\rm
The \emph{Baire category theorem} and \emph{Tao's pigeonhole principle} express central properties of category and measure, and boast numerous equivalences in higher-order RM (see \cite{samBIG2, samBOOK}).  
Restricted to $D$-closed sets, these principles still imply $\NIN_{[0,1]}$, i.e.\ that there is no injection from $[0,1]$ to $\N$, as one readily verifies.  
Some of the aforementioned equivalences should also go through for the restriction to $D$-closed sets and we look forward to exploring this in the future.   
\end{remark}
In conclusion, basic properties of $D$-open sets are equivalent to $\cocode_{0}$, which already boasts many equivalences (\cite{dagsamXI, samBOOK}) and implies $\ATR_{0}$ assuming $\ACAo$.  
However, the $D$-representation only seems like a small variation of the second-order representation of open sets.  By Remarks \ref{beyatr}-\ref{meaca}, lots of future work is possible.    

\subsection{On the strength of the uncountable}\label{mummershines}
In this section, we study the strength of the perfect set theorem and the Cantor-Bendixson theorem when formulated with `countable' instead of `enumerable'.  
As it turns out, these (and similar) theorems cannot yield $\ATR_{0}$ by Theorems \ref{quazulu2}.  The latter follows from Theorem \ref{quazulu}, which in turns is motivated by the following question by Carl Mummert at the 2022 RM-meeting in Paris at the University of Chicago Center.  
\begin{quote}
The principle $\neg\NIN_{[0,1]}$ clearly\footnote{In case there is an injection from $[0,1]$ to $\N$, this mapping must be discontinuous.  
The existence of a discontinuous function (on various domains) implies $(\exists^{2})$ by \cite{kohlenbach2}*{Prop.\ 3.14}.} implies $(\exists^{2})$ and hence $\ACA_{0}$; does it also imply stronger principles, like e.g.\ $\ATR_{0}$?
\end{quote}
First of all, the following theorem provides an answer to Mummert's question.  
\begin{thm}\label{quazulu}
The system $\ACAo+\neg\NIN_{[0,1]}$ cannot prove $\ATR_{0}$. 
\end{thm}
\begin{proof}
We sketch the construction of a model $\mathbf{P}_{0}$ of $\ACAo+\neg\NIN_{2^{\N}}$, which is a modification of the model $\bf P$ of $\Z_{2}^{\omega}+\QFAC^{0,1}+\neg\NIN_{2^{\N}}$ from \cite{dagsamX}. 
We note that $\NIN_{[0,1]}\asa \NIN_{2^{\N}}$ is readily proved (see \cite{samcie22} for a proof). 

\smallskip

Now, Kleene's S1-S9 schemes (\cites{kleeneS1S9, longmann}) provide a model of computation that is intended to capture the following: 
\begin{center}
\emph{an object $X$ is computable from an object $Y$, for $X, Y$ objects of finite type.}
\end{center}
One uses the notation `$\{e\}(Y)=X$' to express that $e\in \N$ encodes a program for computing $X$ from $Y$ via S1-S9.  
The model $\mathbf{P}_{0}$ is then defined as all objects of finite type that are computable via S1-S9 in Kleene's quantifier $\exists^{2}$ from the axiom $(\exists^{2})$.  
Then $\mathbf{P}_{0}$ is a model of $\ACAo$, which is readily proved.   

\smallskip

Now, an important property of S1-S9 is \emph{Gandy selection} (\cite{longmann}), which intuitively expresses that S1-S9-computability is `closed under the Axiom of Choice'.  
To formalise this intuition, consider the following instance of the Axiom of Choice:
\[
(\forall x\in \Gamma)(\exists y\in \Xi)(Z(x, y)=0)\di (\exists \beta)(\forall x\in \Gamma)(\beta(x)\in \Xi \wedge Z(x, \beta(x))=0).
\]
Selection theorems, like Gandy selection, essentially express the following:
\begin{center}
\textbf{If} $\Gamma$ and $\Xi$ consist of S1-S9-computable functions and $Z$ is S1-S9-computable, \textbf{then} the choice function $\beta$ is also S1-S9-computable.
\end{center}
There are of course serious restrictions on the types and other parameters, but the general idea is captured by the above.   As usual in recursion theory, oracles are allowed: 
we can replace `S1-S9-computable' by `S1-S9-computable in $W$', where $W$ is a fixed object of finite type, again with certain restrictions.    

\smallskip

Now consider the following true-by-definition statement about $\bf P_{0}$:
\be\label{chrit}
(\forall f\in 2^{\N}\cap \mathbf{P}_{0})(\exists e\in \N)( \{e\}(\exists^{2})=_{1}f\} ).
\ee
Indeed, \eqref{chrit} just expresses that each element in the Cantor space of $\mathbf{P}_{0}$ is computable (S1-S9) in $\exists^{2}$.   
We may assume that $e\in \N$ is the unique least such number in \eqref{chrit}.  Now apply Gandy selection to \eqref{chrit} to obtain $Y:2^{\N}\di \N$ that is S1-S9-computable in $\exists^{2}$ and satisfies 
\be\label{chrit2}
(\forall f\in 2^{\N}\cap \mathbf{P}_{0})( \{Y(f)\}(\exists^{2})=_{1}f\} ).
\ee
By definition, $Y$ is an element of $\mathbf{P}_{0}$ and is an injection from the Cantor space of $\mathbf{P}_{0}$ to $\N$, as required for $\neg\NIN_{2^{\N}}$.  
We also recall the well-known fact that all reals computable in $\exists^{2}$ are hyperarithmetical (\cite{longmann}). 
However, $\ATR_{0}$ fails in $\textsf{HYP}$, the model of all hyperarithmetical reals (\cite{simpson2}*{V.2.6}).  Hence, $\mathbf{P}_{0}$ satisfies $\ACAo+\neg\NIN_{[0,1]}+\neg\ATR_{0}$, as required for the theorem.   
\end{proof}

\smallskip

Secondly, to be absolutely clear, we define the aforementioned theorems where we stress that `countable' means `injection to $\N$', as always in this paper.  
\begin{princ}[$\CBT$, Cantor-Bendixson theorem]
A closed $C\subset \R$ is the disjoint union of a perfect closed set $P$ and a countable set $S$.  
\end{princ}
\begin{princ}[$\PST$, perfect set theorem]
For any uncountable closed set $C\subset \R$, there exists a perfect set $P$ which is a subset of $ C$.
\end{princ}
We now have the following theorem, to be compared to the fact that second-order $\CBT$ and $\PST$ are respectively equivalent to $\FIVE$ and $\ATR_{0}$ (\cite{simpson2}).  Intuitively, 
these results suggest that $\CBT$ and $\PST$ are weak, both from a second- and third-order point of view.  
\begin{thm}\label{quazulu2} ~
\begin{itemize}
\item The system $\Z_{2}^{\omega}+\QFAC^{0,1}$ plus $\CBT$ and $\PST$ does not imply $\NIN_{[0,1]}$.
\item The system $\ACAo$ plus $\CBT$ and $\PST$ cannot prove $\ATR_{0}$. 
\end{itemize}
\end{thm}
\begin{proof}
For the second item, use Theorem \ref{quazulu} and note that $\neg\NIN_{[0,1]}$ vacuously implies $\CBT$ and $\PST$. 
For the first item, note that $\CBT$ and $\PST$ hold in the model $\bf{P}$ from \cite{dagsamX} as the latter satisfies $\Z_{2}^{\omega}+\QFAC^{0,1}+\neg\NIN_{[0,1]}$. 
\end{proof}
As an informal corollary, we obtain that most theorems about uncountable sets cannot prove $\ATR_{0}$ unless enumerations are somehow included.  This is in stark contrast with the situation in second-order RM, 
where the second-order versions of $\CBT$ and $\PST$ are classified in $\ATR_{0}$ and beyond (\cite{simpson2}).  

\smallskip
\noindent
Finally, the above results do not constitute an isolated incident, as follows.  
\begin{rem}[On Borel and related sets]\label{korelborel}\rm
First of all, the \emph{Lusin separation theorem} states that the existence of certain Borel sets and is equivalent to $\ATR_{0}$ in second-order RM (\cite{simpson2}*{I.11.5}).  
Working in $\ACAo+\neg\NIN_{[0,1]}$, any arbitrary set $E\subset \R$ is Borel (even $\bf F_{\sigma}$) as we have $E=\cup_{n\in \N}E_{n}$ where $E_{n}=\{x\in E: Y(x)=n\}$ is closed and $Y:\R\di \N$ is an injection.  
Hence, the (third-order) Lusin separation theorem (formulated without second-order codes) is provable in the former system and therefore cannot imply $\ATR_{0}$ by Theorem \ref{quazulu}.  

\smallskip

Secondly, Silver's dichotomy theorem is classified in second-order RM as equivalent to $\FIVE$ (\cite{simpson2}*{I.9.4}).
A version of the former theorem is provable in $\ACAo+\neg \NIN_{[0,1]}$ if we use `uncountable' instead of `non-enumerable'.  
Hence, the third-order version of Silver's theorem cannot imply $\ATR_{0}$ by Theorem \ref{quazulu}.

\smallskip

Thirdly, for \emph{any} function $f:\R\di \R$ the discontinuity set $D_{f}$ is $\bf F_{\sigma}$.  This unusually general result goes back to Young (\cite{youngster2}) and some of its formulations cannot imply $\NIN_{[0,1]}$ or $\ATR_{0}$, by the above results.
Indeed, all sets or reals are $\bf F_{\sigma}$ in case $\neg \NIN_{[0,1]}$, which should be obvious by now.      

\smallskip

Fourth, a set is \emph{meagre} if it is the countable union of nowhere dense sets.  Working in $\ACAo+\neg\NIN_{[0,1]}$, any arbitrary set is clearly meagre.  Thus, any theorem about e.g.\ non-meagre sets should generally be provable in the former system and therefore not go beyond $\ATR_{0}$ by the above.  The same holds for sets with the Baire property and strongly measure zero sets (defined using meagreness as in \cite{GMS}).

\smallskip

Fifth, a set is measurable if it is the union of an $\bf F_{\sigma}$ and a measure zero set.  
Working in $\ACAo+\neg\NIN_{[0,1]}$, any arbitrary set is clearly measurable (as it is $\bf F_{\sigma}$).     
Thus, any theorem about e.g.\ non-measurable sets should generally be provable in the former system and therefore not go beyond $\ATR_{0}$ by the above. 
\end{rem}
In conclusion, from the point of view of higher-order arithmetic, theorems about uncountable, non-meagre, or non-measurable sets do not imply $\ATR_{0}$, following Theorem~\ref{quazulu} and Remark \ref{korelborel}. 
Nonetheless, these classes of infinite sets are generally viewed as `exotic', `logically hard', or `out there'.    


\section{Foundational aspects}\label{fasp}
\subsection{Boundaries, then and now}\label{fasp1}
We draw a parallel between the recent history of analysis and the coding of open sets in second-order RM.  

\smallskip

First of all, our starting point is the remarkable Theorem \ref{talkingaboud}.  The latter states that for a closed set $C\subset \R$, we have $C=\textsf{int}(C)\cup \partial C$, i.e.\ a closed set 
can be expressed as the disjoint union of its interior and its boundary.  \emph{Moreover}, we can define the latter sets in $\ACAo$ \emph{and} the interior comes with an RM-code.  
Hence, general closed sets are merely second-order RM-codes (their interior) together with a `small' third-order set (their boundary).  Note that `small' refers to category, as `fat' Cantor sets have non-zero measure but empty interior.  
In particular, Cantor sets are examples of sets that coincide with their boundary. 

\smallskip

Secondly, Hankel makes an incorrect claim in \cite{hankelwoot}*{p.\ 31}, which can be formulated -using our modern definitions- as follows:
\begin{center}
\emph{a bounded function with a dense set of continuity points, is Riemann integrable.  }
\end{center}
A purported proof of the centred statement is based on the (similarly) incorrect claim that nowhere dense sets must have measure zero (\cite{hankelwoot}*{p.\ 26}). 
Smith points out these mistakes in \cite{snutg} and introduces a set $Q\subset \R$ of positive measure that is nowhere dense, i.e.\ a fat Cantor set \emph{avant la lettre}, and shows that any function with discontinuity set $Q$ is not Riemann integrable.  

\smallskip

In conclusion, comparing the previous two paragraphs, we can say the following: since a closed set $C$ is the union of an RM-open set and its boundary $\partial C$, 
studying closed sets via second-order codes ignores the role of (arbitrary) nowhere dense sets.  Ironically, the very same mistake lies at the core of Hankel's incorrect claims and Smith's discovery of the first (fat) Cantor set.  
Put another way, the second-order coding of open sets is reminiscent of pre-Smithian (or: pre-Cantorian) mathematics without arbitrary Cantor sets. 

\subsection{Robustness and coding}\label{fasp2}
We discuss the foundational implications of our results for the coding practice of RM.
We are in good company as Friedman-Simpson discuss this coding practice at length for second-order RM in \cite{fried5}.  
We summarise their views in Section \ref{summac} and provide a critical discussion in Section~\ref{summad}.

\subsubsection{The coding practice and issue of RM}\label{summac}
First of all, RM studies mathematics `as it stands' as opposed to other foundational programs that employ significant modifications.  
The following quotes from \cite{simpson2}*{p.\ 32 and 137} make this precise.  
\begin{quote}
The typical constructivist response to a nonconstructive mathematical theorem is to modify the theorem by adding hypotheses or ``extra data''.
In contrast, our approach in this book is to analyze the provability of mathematical theorems as they stand [\dots]
\end{quote}
\begin{quote}
[\dots] we seek to draw out the set existence
assumptions which are implicit in the ordinary mathematical theorems \emph{as they stand}.
\end{quote}
Variations on this claim are found in \cite{damurm}*{\S10.5.2} and throughout the RM-literature. 

\smallskip

Now, the language of second-order RM is rather frugal: higher-order notions like functions on the reals or open sets of reals are not available directly but need to be represented via second-order `codes'.  
The `coding practice of RM' refers to the whole of these representations.

\smallskip

Secondly, in light of Simpson's quotes, it is then a natural question whether the coding practice of RM changes the (logical strength of the) theorems studied in RM in any significant way.  
This is not a fringe topic in RM: a section titled `The Coding Issue' is devoted to the coding practice of RM in the  `issues and problems in RM' paper \cite{fried5}, 
from which we extract the following quote (\cite{fried5}*{p.\ 135}).
\begin{quote}
Most mathematics naturally lies within the realm of complete separable metric
spaces and continuous functions between them defined on open, closed, compact or
$G_{\delta}$ subsets. This is the central coding issue.
\begin{quote}
PROBLEM. Continuation of the previous problem: Show that 
[the] neighborhood condition coding of partial continuous functions
between complete separable metric spaces is ``optimal". (It amounts
to a coding of continuous functions on a $G_{\delta}$.) 
\end{quote}
We emphasize our view that the handling of the critical coding in $\RCA_{0}$ in [the monograph \cite{simpson2}]
is canonical, but we are asking for theorems supporting this view.  
\end{quote}
The previous quote should be viewed in light of the following\footnote{The authors state in \cite{fried5} that `But infinite sequences of Dedekind cuts are bad.', i.e.\ they are aware of the quote involving Richards.} observation about the representation of the reals in the early days of RM.   
\begin{quote}
Under the old definition [of real number in \cite{simpson3}], it would be consistent with $\RCA_{0}$ that there exists a sequence of real numbers $(x_{n})_{n\in \N}$ such that $(x_{n}+\pi)_{n\in \N}$ is not a sequence of real numbers. We thank Ian Richards for pointing out this defect of the old definition. Our new definition [of real number in \cite{earlybs}], given above, is adopted in order to remove this defect. All of the arguments and results of \cite{simpson3}
remain correct under the new definition. (\cite{earlybs}*{p.\ 129})
\end{quote}
In short, the early representation of `real number' from \cite{simpson3} was not suitable for the development of RM, highlighting the importance of the right choice of definition.  
Similar considerations exist for the definition of continuous function in constructive mathematics (see \cites{waaldijk, vandebrug}), i.e.\ this situation is not unique to RM. 

%
%
%
%

\smallskip

In conclusion, we hope we have convinced the reader that the issues surrounding the coding of higher-order objects are taken seriously in second-order RM.  

\subsubsection{Robustness and coding}\label{summad}
We discuss the foundational implications of our results, especially the robustness of the second-order coding of open sets. 

\smallskip

First of all, we quote the full passage by Montalb\'an on robustness. 
\begin{quote}
Even though we now know of many theorems that are not equivalent to any of the big five systems, we would still claim that the great majority of the theorems from classical mathematics are equivalent to one of the big five. This phenomenon is still quite striking. Though we have some sense of why this phenomenon occurs, we really do not have a clear explanation for it, let alone a strictly logical or mathematical reason for it. The way I view it, gaining a greater understanding of this phenomenon is currently one of the driving questions behind reverse mathematics.  To study the big five phenomenon, one distinction that I think is worth making is the one between
robust systems and non-robust systems. A system is \emph{robust} if it is equivalent to small perturbations
of itself. This is not a precise notion yet, but we can still recognize some robust systems. All the big
five systems are very robust.  (see \cite{montahue}*{p.\ 432})
\end{quote}
Secondly, in light of our above results, the coding of open sets in RM can only be called \emph{non-robust} as the $D$-representation constitutes only a small variation, but yields much stronger results. 
Indeed, the Urysohn lemma for $D$-closed sets yields $\ATR_{0}$ when combined with arithmetical comprehension as in $\ACAo$.  By contrast, $\ATR_{0}$ is \emph{much} stronger than
the Urysohn lemma formulated with second-order codes, which is provable in $\RCA_{0}$ by \cite{simpson2}*{II.7}.  Hence, the second-order representation of open sets hardly seems ``optimal'' in any reasonable sense.         

\smallskip

Thirdly, we stress that while the $D$-representation challenges the suitability of the second-order representation of open sets, our results do not change the observation that the Big Five are robust, an observation made in Montalb\'an's above quote.  Moreover, it may well be that there is no representation of open sets that makes everyone happy.  However, whatever representation one favours, their properties do seem to yield equivalences for the Big Five, as discussed in Remarks \ref{tiritomba} and \ref{beyatr}.       

\begin{ack}\rm
We thank Carl Mummert and Jeff Hirst for the interesting discussion related to Section \ref{mummershines}.  
The ideas for this paper, namely studying variations of representations, is inspired by \cite{pischkeetal}. 
\end{ack}

\end{document}